\documentclass[11pt]{amsart}
\usepackage{amsfonts}
\usepackage{mathrsfs}
\usepackage{graphicx}
\usepackage{verbatim}
\usepackage{color}
\usepackage{amsmath,amsthm,amssymb}
\usepackage{array,longtable,calc}
\usepackage{caption}
\usepackage{subcaption}
\usepackage{enumitem}
\usepackage{geometry}
\usepackage{hyperref}%

\newtheorem{theorem}{Theorem}[section]
\newtheorem*{theorem*}{Main Theorem 1}
\newtheorem*{theorem**}{Main Theorem 2}
\newtheorem{lemma}[theorem]{Lemma}
\newtheorem{corollary}[theorem]{Corollary}

\newtheorem{conjecture}[theorem]{Conjecture}
\theoremstyle{definition}
\newtheorem{definition}[theorem]{Definition}

\theoremstyle{remark}
\newtheorem{remark}[theorem]{Remark}

\newtheorem{question}[theorem]{Question}

\newtheorem{notation}[theorem]{Notation}

\numberwithin{equation}{section}

\newcounter{casenum}
  \newenvironment{caseof}{\setcounter{casenum}{1}}{\vskip.5\baselineskip}

\newcommand\addtag{\refstepcounter{equation}\tag{\theequation}}
\makeatletter
\def\localbig#1#2{%
  \sbox\z@{$\m@th#1
    \sbox\tw@{$#1()$}%
    \dimen@=\ht\tw@\advance\dimen@\dp\tw@
    \nulldelimiterspace\z@\left#2\vcenter to1.2\dimen@{}\right.
  $}\box\z@}

\newcommand{\divides}{\mathrel{\mathpalette\dividesaux\relax}}
\newcommand{\ndivides}{\mathrel{\mathpalette\ndividesaux\relax}}

\newcommand{\dividesaux}[2]{\mbox{$\m@th#1\localbig{#1}|$}}
\newcommand{\ndividesaux}[2]{%
  \mkern.5mu
  \ooalign{%
    \hidewidth$\m@th#1\localbig{#1}|$\hidewidth\cr
    $\m@th#1\nmid$\cr%
  }%
}
\makeatother

\author{Dheeraj Kulkarni}
\address{Department of Mathematics \\ Indian Institute Of Science Education and Research, Bhopal.}
\email{dheeraj@iiserb.ac.in}

\author{T. V. H. Prathamesh}
\address{Institute Of Mathematical Sciences, HBNI, Chennai}
\email{prathamesh@imsc.res.in}

\begin{document}
\title[Legendrian Racks]{On Rack Invariants of Legendrian Knots }
\begin{abstract}
In this article, we introduce rack invariants of oriented Legendrian knots in the 3-dimensional Euclidean space endowed with the standard contact structure, which we call Legendrian racks. These invariants form a generalization of the quandle invariants of knots. These rack invariants do not result in a complete invariant, but detect some of the geometric properties such as cusps in a Legendrian knot.

In the case of topologically trivial Legendrian knots, we test this family of invariants for its strengths and limitations. We further prove that these invariants form a  natural generalization of the quandle invariant, by which we mean that any rack invariant under certain restrictions is equivalent to a Legendrian rack. 

The axioms of these racks are expressible in first order logic, and were discovered through a series of experiments using an automated theorem prover for first order logic. We also present the results from the experiments on Legendrian unknots involving automated theorem provers, and describe how they led to our current formulation. 

\end{abstract}
\maketitle

\section{Introduction}
Racks are non-associative algebraic structures, which proved to be a source of invariants of various classes of knots. Idempotent racks, called quandles, are an invariant of (tame) knots under ambient isotopy. \cite{Joyce} Generalizations of this invariant were subsequently derived for framed links and virtual knots \cite{FennRourke} \cite{VirtRacks}. 

Distinguishing racks and quandles is in practice very difficult. Advances in automated theorem proving for first-order logic, enabled use of computers to distinguish non-triviality of racks and quandles. This led to the possibility of using automated theorem provers for unknot recognition \cite{Fish}. 

The above mentioned computations were made possible by the fact that definitions of racks and quandles are expressible in first-order logic. On these lines, we ask the following question:
\begin{question}
Are there any rack invariants of  Legendrian knots which are axiomatisable in first order logic? 
\end{question}

In this article, we introduce a family of racks, which we call  \emph{Legendrian racks} in Defintion \ref{Leg_racL_defn}. To every front projection $D_L$ of a Legendrian knot $L$, we associate a family of Legendrian racks indexed by natural numbers, $\mathcal{LR}(D_L) = \{ LR_n(D_L) \big\}_{n \in \mathbb{N}}$. The main result of this article establishes the invariance of the family under Legendrian Reidemeister moves. Thus, $\mathcal{LR}(D_L) $ defines an invariant of $L$ up to Legendrian isotopy. 

\begin{theorem*}
Each Legendrian rack $LR_n(D_L) $ associated to a front projection $D_L$ of an oriented Legendrian knot $L$ is independent of the choice of front projection of $L$. Therefore, $LR_n(D_L) $ is an invariant of $L $ up to isotopy via a family of oriented Legendrian knots.
\end{theorem*}
Each Legendrian rack $LR_n(D_L) $ associated to a front projection $D_L$ of oriented Legendrian knot $L$ is independent of the choice of front projection of $L$. Therefore, $LR_n(D_L) $ is an invariant of $L $ up to isotopy via a family of oriented Legendrian knots.
We further prove that Legendrian racks corresponding to topologically distinct Legendrian knots, are distinct. An invariant of Legendrian knots is of any interest, only if it distinguish Legendrian knots which correspond to the same topological knot. To illustrate this, we show that Legendrian racks  distinguish a large class of Legendrian knots corresponding to the topological unknot.

\begin{theorem**}
Let $L_1$ and $L_2$ be two topologically trivial Legendrian knots. If there exists an odd prime $p$ and a positive integer $k$ such that: 
\begin{itemize}
\item $p^k \divides s_{min}(L_1)$
\item $p^k \ndivides s_{min}(L_2)$
\end{itemize}
Then $\mathcal{LR}(L_1) \neq \mathcal{LR}(L_2)$.
\end{theorem**}
The function $s_{min}$ on Legendrian knots equals $2 \mid tb(L) \mid$, where $tb$ is the Thurston-Bennequin number. We further prove that two Legendrian unknots with the same Thurston-Bennequin number, map to isomorphic Legendrian racks, which further strengthens the link with Thurston-Bennequin number. Thurston-Benequinn number of a Legendrian knot contains information only about the total number of cusps in the front projection, and does not contain any information about the difference in the number of up and down cusps in the front projections of the knot. This motivates us to ask the following question:
\begin{question}
Does there exist a rack invariant of oriented Legendrian knots, defined in terms of the front projection, which contains information about the difference in the number of up and down cusps in the front projections of the Legendrian knot?
\end{question}

We show in Theorem \ref{LPRimpliesLR} that under some assumptions about the expressibility of the definition, such an invariant will be equivalent to Legendrian racks. The proof is derived by considering a set of universally quantified axioms expressible in first-order logic that such an invariant should satisfy, and by further showing that such an invariant reduces to a Legendrian rack. Theorems \ref{LPRimpliesLR} to \ref{cusp_equiv_nth_power} deal with the above-mentioned   results.

The work has an experimental component to it, as discovery of many of the results described above were largely aided by the use of an automated theorem prover - Prover9 and Mace4. This includes generation of proofs, counter-examples and even guiding the authors towards more elegant definitions.

This work should be of interest to contact topologists studying Legendrian knots, as well as those seeking newer applications of racks and quandles. It might also be of potential interest to those interested in applications of automated reasoning in mathematics. 

The article is organized as follows:
 We have added a brief recall of all the relevant notions from contact topology keeping in mind the wide audience the article may reach. Thus, a reader with background in contact topology may safely skip it. Section \ref{Leg_racL_defn} contains the definition of Legendrian racks and it shows how to 
associate Legendrian racks to
oriented Legendrian knots using front projections. Section \ref{Leg_racL_defn} also contains the proof of invariance of Legendrian racks under Legendrian isotopy. Section \ref{Leg_unknot} discusses the ability of Legendrian racks to distinguish topologically equivalent Legendrian knots by considering the case of Legendrian unknots.
Section \ref{alt_defn}  contains the metamathematical results, which explain why Legendrian racks are the appropriate generalization. Section \ref{fotp}, we discuss the experimental aspect of the work using automated theorem provers. Section \ref{concl} contains further questions and conclusion.

\vspace{3mm}
{\bf Acknowledgements:} The authors would like to thank Indian Statistical Insitute, Kolkata, India, Harish-Chandra Research Institute, Allahabad, India and Institute of Mathematical Sciences, India for their support when this work was carried out. The first author wishes to gratefully acknowledge the support from IISER Bhopal through the grant IISERB/INS/MATH/2016091 during the final part of this work.

\subsection{The standard Contact Structure on $\mathbb{R}^3 $}
\begin{definition}
Let $(x, y, z) $ denote the standard coordinate system on $\mathbb{R}^3$. The standard contact structure on $\mathbb{R}^3$ is a 2-plane field given by the $\text{ker} (dz-ydx) $. We denote it by $\xi_{std} $.
\end{definition}
%

In general, contact structures are defined as nowhere integrable hyperplane fields of 
codimension one on a manifold. The study of contact structures, in recent times, has seen very rapid progress. Especially 
in dimension three, contact structures have been understood to great extent. Further, the area of contact topology is 
interlinked with low dimensional topology. For a comprehensive introduction to contact structures reader is referred to \cite{Geiges} 
.

\subsection{Legendrian knots}
\begin{definition}
A (smooth) knot $L$ is in $\mathbb{R}^3$ is called \textit{Legendrian} if at every point $p \in L $, the tangent space $T_p L $
is a subspace of $\xi_{std}|_p $.
\end{definition}\label{Lknot}
 In other words, $L$ is tangent to the contact planes at all points on $L$. 
\begin{definition} 
 We say that Legendrian knots $L_1$ and $L_2$ are \textit{Legendrian isotopic} if there is an isotopy taking $L_1$ to $L_2$ through family of Legendrian knots. 
\end{definition}

In this article, we will be concerned with \emph{oriented} Legendrian knots. We say that Legendrian knot is oriented if there exists an orientation on the knot. Equivalence of two oriented Legendrian knots is defined through an orientation preserving isotopy.
Legendrian knots are studied by looking at their {\em front projections}. Now, we explain the idea of front projection. Let $\varphi : \mathbb{S}^1 \rightarrow \mathbb{R} ^3 $ be a Legendrian knot. We write $\varphi = (x(\theta), y(\theta), z(\theta)) $. Since $Im (\varphi) $ is Legendrian, we obtain the following equation.
$$\dot{z}(\theta) - y(\theta) \dot{x}(\theta)=0  $$
One may rewrite the above as $ y(\theta) = \frac{dz}{dx}|_{\theta} $ provided $(x(\theta, z(\theta)) $ has no vertical 
tangency. This implies that the coordinate $y $ of the knot $Im(\varphi) $ can be recovered from slope of its $x-z $ 
projection as long as there is no vertical tangency. A generic $ x-z$ projection will have finitely many cusp singularities
to replace vertical tangencies and only one type of crossing appears as shown in the Figure \ref{Leg_trefoil}. 

\begin{figure}
\centering
\scalebox{0.5}{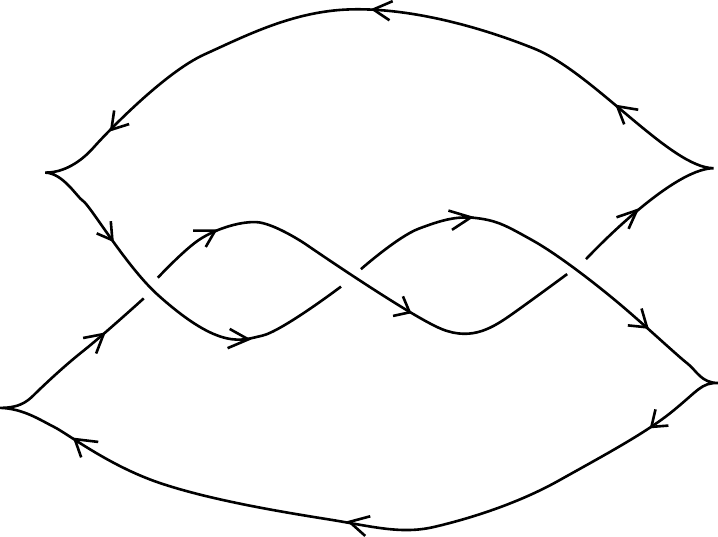}
\caption{A front projection of Legendrian Trefoil.}\label{Leg_trefoil}
\end{figure}

The Legendrian isotopy can be understood via Reidemeister moves in the front projection. We will refer to these moves as {\em Legendrian Reidemeister moves}. Figure \ref{LR-moves} shows three basic types of Legendrian Redemeister moves which are analogous to the Reidemeister moves in classical knot theory. One can obtain all Legendrian Reidemeister moves by rotating each diagram by $180 $ degrees about all the coordinate axes. Further one may add orientations to these diagram.\\
\begin{figure}[h!]
\centering
\scalebox{0.5}{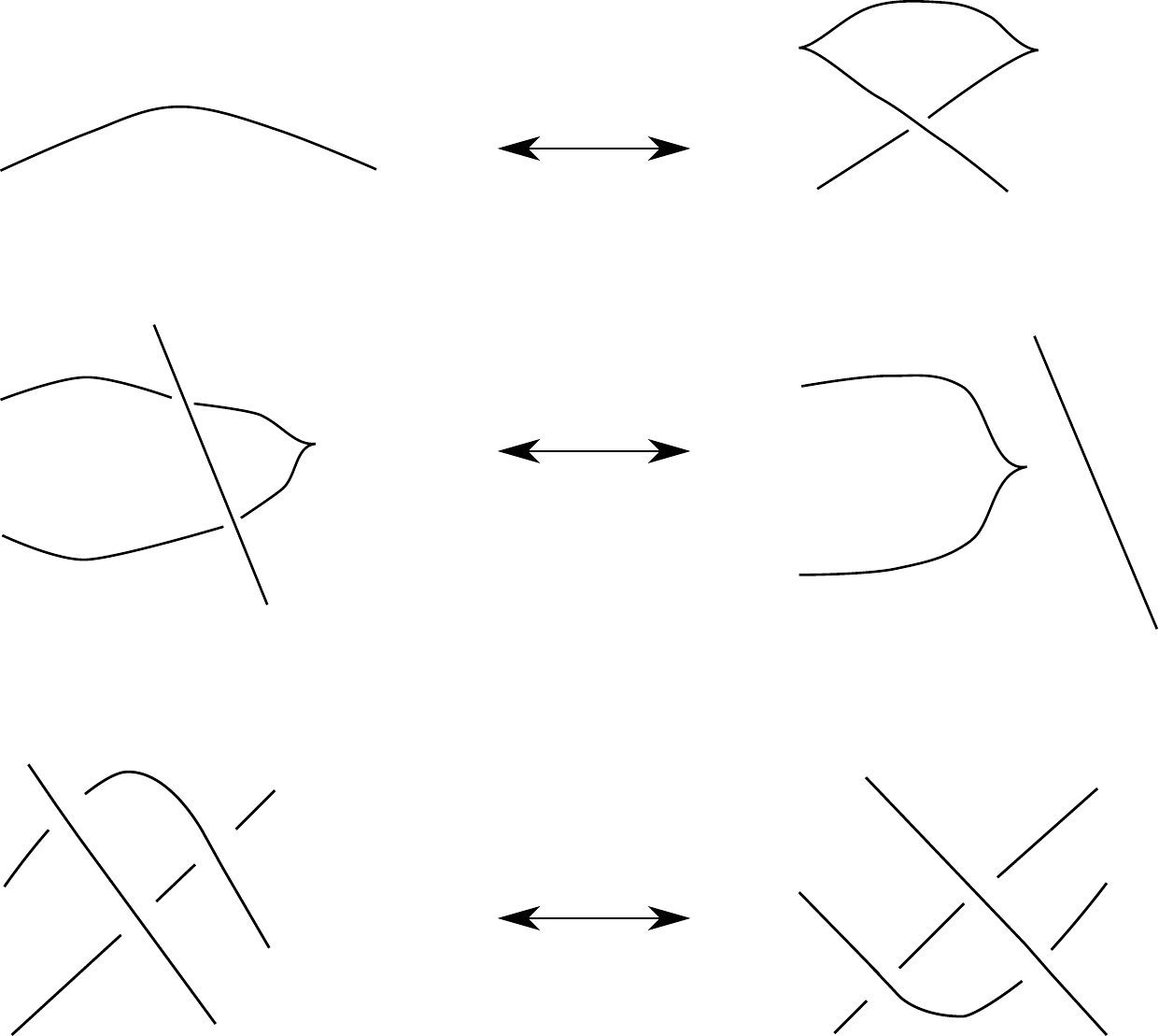}
\caption{Legendrian Reidemeister moves}\label{LR-moves}
\end{figure}




The classification of Legendrian knots up to Legendrian isotopies is a challenging problem. There are classical invariants 
such as rotation number, Thurston-Bennequin number for knots in $\mathbb{R}^3 $. However, these invariants are not 
sufficient to classify Legendrian knots. contact homology, Chekanov's DGA are some recent tools which are finer than the classical invariants (See \cite{Etnyre}). Eliashberg and Fraser proved in \cite{Eliash} that in the case of topologically trivial Legendrian unknots rotation number 
together with Thurston-Bennequin number completely determines the Legendrian unknot.

\section{Legendrian racks and Legendrian knots}\label{Leg_racL_defn}

The challenge in defining racks for oriented Legendrian knots lies in defining a suitable mathematical 
structure for encoding information about cusps. When associating a quandle to a knot 
diagram, we assign a generator of the quandle to each strand of the knot, and crossings give rise to quandle relations.
In the case of Legendrian knots, we associate a family of racks to a Legendrian knot. Such an association is obtained by  
redefining strands to refer to (continuous) segments of a knot diagram that either begin or end in a crossing or a  cusp. 
Each crossing and cusp give rise to new relations on a rack. We prove that this structure is an invariant of 
oriented Legendrian knots upto orientation preserving Legendrian isotopy. We begin by revisiting the definition of a rack.

\begin{definition} A rack is set $R$ with binary operations $(* , /)$, such that every $x$, $y$ and $z$:
\begin{enumerate}
\item \textbf{Self-Distributivity:} $(x*y)*z = (x*z)*(y*z)$ and $(x/y)/z = (x/z)/(y/z)$ .
\item \textbf{Existence of Right Inverse:} $(x*y)/y = x$, and  $(x/y)*y = x$.
\end{enumerate}
\end{definition}

\begin{notation} We will use $(R, *)$ to denote a rack, since $/$ can be defined in terms of $*$.
\end{notation}

\begin{definition}: The power function on a rack $(R, *)$ is the function $pow$: $ \mathbb{Z} \backslash \{0,-1\} \times  R \rightarrow R$, such that:
\[ pow (n, x) = 
\begin{cases}
    x,& \text{if } n =  1\\
    ((pow (n-1,x))*x), &\text{if } n > 1\\
  (x/x), &\text{if } n = -2\\
  ((pow (n+1,x))/x) &\text{if } n < -2\\
\end{cases}
\]
\end{definition}

\begin{notation}  We will use $x^n$ to denote $pow(n,x)$.\end{notation}

To illustrate the definition above, $x^4$ in the above notation would expand to $(((x*x)*x)*x)$, while $x^{-3}$ would expand to $((x/x)/x)$.

\begin{definition}[Legendrian rack]\label{LracL_defn} Let $n \in \mathbb{N}$. An $n$-Legendrian rack $(LR_n, *)$ is a rack such that:
\[\forall x.\ x^{2n+2} = x\]
\end{definition}

\begin{remark} $0$-Legendrian rack is the quandle.
\end{remark}

\subsubsection{Examples:}We have the following example of a finite $n$-Legendrian rack. Consider the rack $(C_k, *)$ defined as follows:
\begin{itemize}
\item  $C_k = \{0,\ 1,\ 2,\ ...,\ k-1\ \}$.
\item $i*j\ =\ (i+1)\ mod\ k$.
\end{itemize}

\begin{notation} We will use the terms $LR_n$ and $C_k$ to refer to the associated racks. 
\end{notation}

We have the following proposition:
\begin{theorem}\label{fin-L-rack}
If $k > 1$ and it divides $2n+1$, then $C_k$ is an $n$-Legendrian rack.
\end{theorem}
\begin{proof}
By the definition of rack operation on $C_k $, we have $i^{2n+2} \equiv i+2n+1\ (\text{mod}\ k)$ in $C_k $.
Since $k\ |\ (2n+1) $, we get $i^{2n+2} \equiv i\ (\text{mod} \ k) $. Thus, $C_k $ is an $n$-Legendrian Rack.

\end{proof}

Given a front projection $D_L$ of an oriented Legendrian knot $L$, one can associate an indexed family of  $n$-Legendrian racks $\mathcal{LR}(D_L) = \big( LR_n(D_L) \big)_{n\in\mathbb{N}}$  to the front projection $D_L$ in the following fashion:


\begin{itemize}

\item We use the term strands to denote the connected segment of an arc in an oriented Legendrian knot diagram (in the front projection), which begin and end at a crossing or a cusp. A strand thus does not contain any cusps. A Legendrian knot diagram in the front projection can be pictured as a union of strands intersecting at cusps.

\item Each crossing gives rise to a new strand and a new relation along the orientation. The resultant relations are the same as in quandles and classical knots.

\item Each cusp also gives rise to a new strand and a new relation. If the strand `b' is related to the strand `a' by a cusp along the direction of orientation (See Figure \ref{cusp_relation}), then we have the equation 
\[b = a^{n+1}\]
\begin{figure}
\centering
\scalebox{0.75}{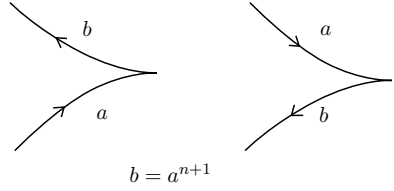}
\caption{Relations corresponding to cusps}\label{cusp_relation}
\end{figure}
\end{itemize}

The above can be illustrated by considering the following example of the Legendrian trefoil in Figure \ref{Leg_trefoil_with_labels}.
\begin{figure}[h!]
\centering
\begin{minipage}{.6\textwidth}
\centering
\scalebox{0.5}{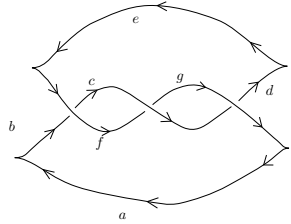}
\captionof{figure}{Legendrian Trefoil with labels}\label{Leg_trefoil_with_labels}
\end{minipage}
\begin{minipage}{.3\textwidth}
\centering
\begin{itemize}
\item $ b= a^{n+1} $.
\item $ c= b \ast f $
\item $ g = f \ast c $
\item $ d = c \ast g $
\item $ e =d^{n+1} $
\item $ f = e^{n+1}$
\item $ a = g^{n+1} $
\end{itemize}
Relations associated to crossings and cusps.
\end{minipage}
\end{figure}

\subsection{Invariance of $n$-Legendrian racks}
We prove the following intermediary lemmas, before we proceed to show the invariance. The following lemmas hold true in any rack $(R, *)$.

\begin{lemma}\label{LR1-1}
\[\forall x, y \in R.\ \ x * (y * y) = x * y\]
\end{lemma}
\begin{proof}
We have the following equality from self-distributivity axiom,
\[((x/y)*z)*y) = ((x/y)*y)*(z*y) \addtag \]
Using the right inverse axiom,
\[((x/y)*y)*(z*y) = x * (z * y)\addtag \]
For $z = y$, the above equalities reduce to,
\[x * y = x * (y * y) \addtag \]
\end{proof}

\begin{lemma}\label{LR1-2}
\[\forall a\in R.\ \forall k \in \mathbb{Z}_+.\ \ a^{k+1} = (a^k * a^k)\]
\end{lemma}
\begin{proof}
We prove it by induction on $k$. For $k$ = 1, it follows trivially. For $k = n + 1$, We have
\[a^{n+1} = (a^n * a) \addtag \]
From induction hypothesis,
\[ (a^n * a) = (a^{n-1} * a^{n-1})*a \addtag \]
From self-distributivity of racks, 
\[ a^{n+1} =(a^{n-1} * a)* ( a^{n-1}*a) \addtag \]
which reduces to,
\[ a^{n+1} =(a^n)* (a^{n}) \addtag \]
\end{proof}

\begin{corollary}\label{LR1-2-2}
\[ \forall x, y  \in R.\ \forall k \in \mathbb{N}.\ \  x * y^{k+1} = x * y\]
\end{corollary}
\begin{proof}
We have the following equality by lemma \ref{LR1-2}.
\[y^{k+1} = y^{k} * y^{k}\]
Which implies that 
\[x * y^{k+1} = x * (y^{k} * y^{k})\]
By applying lemma \ref{LR1-1} to R.H.S.,
\[x * y^{k+1} = x * y^{k}\]
which recursively reduces to the stated goal.
\end{proof}

\begin{lemma}\label{LR1-2.5}
\[\forall a\in R.\ \forall b\in R.\ \forall k \in \mathbb{N}.\ \ (a * b)^{k} = a^{k} * b\]
\end{lemma}
\begin{proof}
We prove by induction. The $k = 1$ case trivially follows. Let us assume the result holds true for $m$. We prove it for $k = m +1$.\\
From induction hypothesis,
\[(a * b)^{m} = (a^m * b)\]
By definition of power,
\[(a * b)^{m+1} = (a * b)^{m}*(a*b) \addtag \]
By applying induction hypothesis,
\[(a * b)^{m+1} = (a^{m} * b)*(a*b) \addtag \]
By using self-distributivity on R.H.S,
\[(a * b)^{m+1} = (a^{m} * a)*b \addtag \]
This simplifies to,
\[(a * b)^{m+1} = (a^{m+1}*b)\].
\end{proof}

\begin{lemma}
\[\forall a\in R.\ \forall k \in \mathbb{N}.\ \  {(a^{k+1})}^{k+1} = a^{2k+1}\]
\end{lemma} \label{LR1-3}
\begin{proof}
We prove by induction on $k$, For $k = 0$, this holds true trivially. We assume that the result holds for $k \leq m-1$, and proceed to show it for $k = m$. 

Let $b = a^m$,
Then we have the following equality,
\[(a^m*a)^{m+1} = (b*a)^{m+1} \addtag \]
Then from lemma \ref{LR1-2.5} we have,
\[(b*a)^{m+1} = (b^{m+1} * a) \addtag \]
We know that,
\[b^{m+1} = b^m * b\]
By substituting the above in 2.12 and by further substituting for $b$ we have,
\[(b^m * b) * a = ((({a^m})^m) * (a^m)) * a \addtag \]
We know from induction hypothesis that,
 \[({a^m})^m  = a^{2m-1} \addtag \]
Substitution the above in 2.13, 
 \[(b^m * b) * a = (a^{2m-1} * a^m) * a \addtag \]
By applying corollary \ref{LR1-2-2} on R.H.S,
 \[(b^m * b) * a = (a^{2m-1} * a) * a \addtag \]
which reduces to the desired equation.
 \end{proof}
 \begin{remark}
The above lemma indicates that the presentations of the Legendrian rack associated to Legendrian knots are sensitive to orientation, else by going through the same cusp again, we should get ${(a^{n+1})}^{n+1} = a$.
 \end{remark}

The following holds true in the Legendrian rack $LR_n$.

\begin{lemma}\label{LR1-2-3}
\[\forall a,b \in LR_n.\ (a^{2n+1} = b^{2n+1}) \longrightarrow (a = b) \]
\end{lemma}
\begin{proof}
Without loss of generality, we fix $a$ and $b$ and assume that $a^{2n+1} = b^{2n+1}$. Then we have following equality for all $x \in R$,
\[x * a^{2n+1} = x * b^{2n+1}\]
By applying corollary \ref{LR1-2-2} on both sides, we get
\[x * a  = x * b\]
For $x = a^{2n+1}$,
\[a^{2n+1} * a  = a^{2n+1} * b\]
By using the assumption,
\[a^{2n+1} * a  = b^{2n+1} * b\]
We thus obtain the following equality using the $n$-Legendrian rack axiom,
\[a = b\]
\end{proof}

 \begin{lemma} $\forall a \in LR_n.\ a^{n+1} = a^{-(n+2)}$
\end{lemma}\label{LR1-4}
\begin{proof}
We obtain the above by consecutive right multiplication of $a$ on both sides. The above equation reduces to
\[a^{2n+2} = a\]
\end{proof}

We restate and prove the main theorem 1, as follows: 

\begin{theorem*}
Each Legendrian rack $LR_n(D_L) $ associated to a front projection $D_L$ of an oriented Legendrian knot $L$ is independent of the choice of front projection of $L$. Therefore, $LR_n(D_L) $ is an invariant of $L$ up to isotopy via a family of oriented Legendrian knots.
\end{theorem*}\label{lracL_invariance}

\begin{proof}
Recall that any two front projections of $L$ are related by Legendrian Reidemeister moves. Hence, it is enough to prove the invariance of $LR_n(D_L) $ each Legendrian Reidemeister move. 

\begin{itemize}
\item\textbf{Legendrian Reidemeister Move 1:}
\end{itemize}
 Legendrian Reidemeister move 1 consists of four diagrams of which two are shown in Figure \ref{LR1_diag}. Rest of the cases can be derived by vertically rotating the two diagrams by 180 degrees. 
One can check that vertical rotation by 180 degrees does not alter the induced rack relations in the presentation of the Legendrian racks.  Thus it suffices to consider the first 2 cases.
\begin{figure}[h!]
\centering
\scalebox{.75}{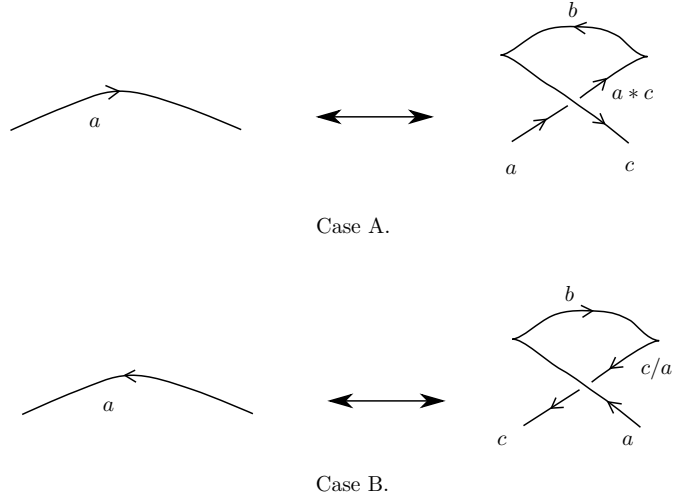}
\caption{Legendrian Reidemeister move of type 1}\label{LR1_diag}
\end{figure}
\begin{caseof}{\textbf{Case A:} }{
Consider the first figure, in Figure 6. We label the figure on the left as $D_L$ and on the right as $D'_L$. 
Invariance of the $n$-Legendrian rack under Legendrian Reidemeister move, is equivalent to showing that $LR_n(D_L) = LR_n(D'_L)$. The presentations corresponding to the front projection for $LR_n(D_L)$ and $LR_n(D'_L)$ can be described as follows:
\begin{enumerate}
\item $LR_n(D_L) = \langle G\cup \{a\}\ \ |\ \ R\cup  R_a\cup R'_a \rangle$ 
\item $LR_n(D'_L) = \langle\ G\cup \{a,b,c\}\ \ |\ \ R\cup R_a\cup R'_c\cup \{ b = (a*c)^n,\ c = b^{n+1}\} \rangle$
\end{enumerate}

$LR_n(D'_L)$ and $LR_n(D_L)$ differ in the following respect in terms of the presentation associated to the front projection:
\begin{enumerate}
\item $LR_n(D'_L)$ consists of two additional generators $b$ and $c$.
\item $R_a$ and $R'_a$ refer to the relations induced by crossings and cusps involving the strand corresponding to $a$, which occur before and after the point on which the Legendrian Reidemeister Move 1 is applied. In $LR_n(D'_L)$ $R'_a$ is replaced by $R'_c$, in which every instance of $a$ in $R'_a$ is replaced by $c$.  
\item The following additional relations are present in $LR_n(D'_L)$: 
\begin{itemize}
\item $b = (a*c)^{n+1}$.
\item $c = b^{n+1}$.
\end{itemize}
\end{enumerate}

The above implies that in $LR_n(D'_L)$,
 \[c = {\big( (a*c)^{n+1}\big) }^{n+1}\]
Which implies from lemma \ref{LR1-3}
 \[c = (a * c)^{2n+1}\]
Which further implies from lemma \ref{LR1-2.5}
 \[c = a^{2n+1} * c\]
From the definition of an $n$-Legendrian rack it follows, 
 \[c^{2n+2} = a^{2n+1} * c\]
From the right-inverse axiom of racks, 
 \[c^{2n+1} = a^{2n+1}\]
From lemma \ref{LR1-2-3},
 \[c = a\]
By substituting this in the equation for $b$,
 \[ b = (a * a)^{n+1}\]
 
From the use of lemma \ref{LR1-2.5},
\[b = a^{n+1} * a\]
\[ b = a^{n+2}\]

Thus the generating set of $LR_n(D'_L)$ is equal to the generating set of $LR_n(D_L)$. The two $n$-Legendrian racks are equal if the generating set of relations also are the same. Since $a$ is equal to $c$, the relations involving $c$ in $LR_n(D'_L)$, reduce to the corresponding relations in $LR_n(D_L)$. The generating set of relations in $LR_n(D'_L)$ after substituting for $b$ and $c$ in terms of $a$, contain the following additional relations along with $R \cup R_a \cup R'_a$.
\begin{itemize}
\item $a^{n+2} = (a*a)^{n+1}$.
\item $(a^{n+2})^{n+1} = a$.
\end{itemize}
It suffices to show that these relations are satisfied in $LR_n(D_L)$.
From lemma \ref{LR1-2.5} it follows that,
\[(a * a)^{n+1} = a^{n+1} * a\]
which implies that first of the above-mentioned additional relations. We prove the second relation on the following lines. From the definition of power it follows that,
\[({a^{n+2}})^{n+1} = (a^{n+1}*a)^{n+1}\]
By applying lemma \ref{LR1-2.5} to R.H.S,
\[(({a^{n+2})}^{n+1}) = ({a^{n+1}})^{n+1}*a \addtag \]
From lemma \ref{LR1-3},
\[(({a^{n+1}})^{n+1}) = a^{2n+1}\]
By substituting in 2.17,
\[(({a^{n+2}})^{n+1})  = a^{2n+2}\]
By applying the $n$-Legendrian rack axiom,
 \[(({a^{n+2}})^{n+1})  = a\]

Thus the Legendrian Reidemeister Move of type (1) under consideration preserves the associated $n$-Legendrian rack.}\\

{\textbf{Case B:}} {Now we consider the Legendrian Reidemeister move of type (1), with the opposite orientation, as described in the Figure 6. Analogous to the above case, the presentation associated to the front projection of the Legendrian rack $LR_n(D'_L)$ contains additional generators $b$ and $c$, relations obtained by selectively replacing $a$ by $c$ and the following additional relations:
\begin{itemize}
\item $b = a^{n+1}$
\item $c/a = b^{n+1}$
\end{itemize}
The second relation above can be rewritten using the right inverse axiom of racks,
\[ c = (b^{n+1})*a\]
By substituting for $b$,
\[c = {(a^{n+1})}^{n+1}*a\]
which implies from lemma \ref{LR1-3} that,
\[c = a^{2n+1}*a\]
From the definition of $n$-Legendrian rack,
\[c = a\]
Since both $b$ and $c$ can written in terms of powers of $a$, the generator set in $LR_n(D'_L)$ is the same. Further since $a = c$, the relations involving $c$ are the same as the corresponding relations involving $a$ in $LR_n(D_L)$. The additional relations satisfied in $LR_n(D'_L)$, after substituting in terms of $a$ reduces to:
\begin{itemize}
\item $a/a = (a^{n+1})^{n+1}$.
\end{itemize}
This is equivalent to proving that,
\[a = (a^{n+1})^{n+1}* a\]
This holds true in $LR_n(D_L)$, by using the lemma \ref{LR1-3} and the $n$-Legendrian rack axiom. 

One can check that any other variant of the Legendrian Reidemeister move of type I reduces to the one of the above cases, in terms of correspondence with racks. This $n$-Legendrian racks corresponding to oriented Legendrian knots are invariant under Legendrian Reidemeister moves of type 1.}
\end{caseof}

\begin{itemize}
\item \textbf{Legendrian Reidemeister Move 2.}
\end{itemize}
Legendrian Reidemeister moves of type 2 consist of the four moves listed in figure 7.

\begin{figure} \label{Fig-LR2}
\centering
\scalebox{0.6}{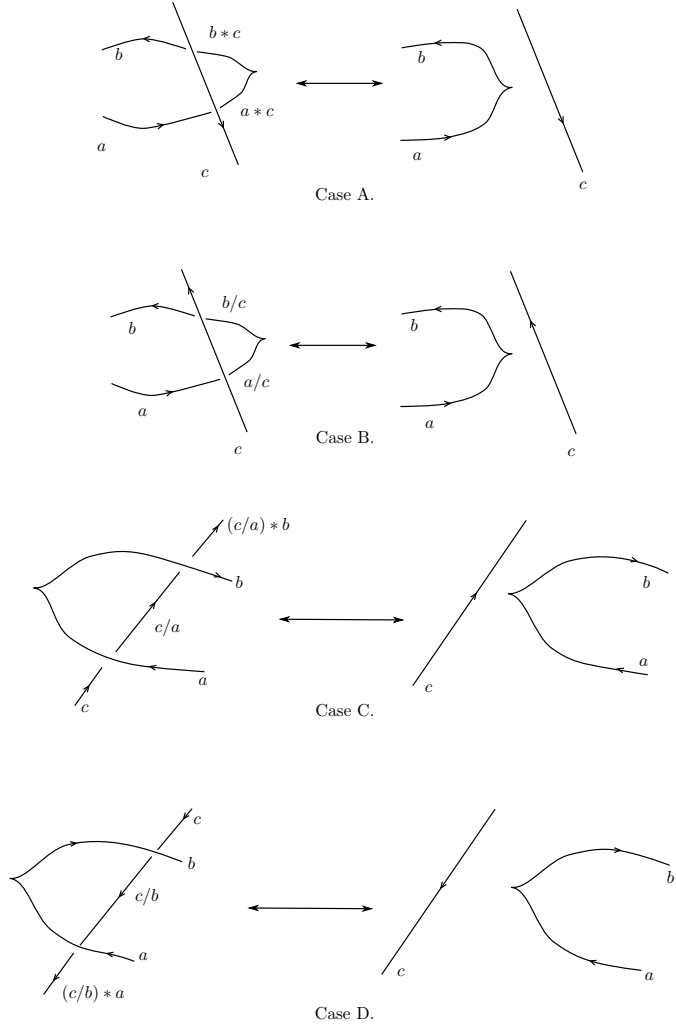}
\caption{Legendrian Reidemeister move of type 2}
\end{figure} 

Analogous to the earlier case, we use the terms $D_L$ and $D'_L$ to denote the diagrams on left and right, in each of the above listed diagrams. Proofs for each of the cases are listed below:\\
\begin{caseof}
{\textbf{Case A:}}

 {Presentations corresponding to the front projection for $LR_n(D_L)$ and $LR_n(D'_L)$ differ in the following respect:
\begin{enumerate}
\item The generating sets are the same.
\item $LR_n(D_L)$ contains the relation $b = a^{n+1}$. In $LR_n(D'_L)$, the relation is replaced by the relation $b * c = (a * c)^n$.
\end{enumerate}
Showing equivalence of the two $n$-Legendrian racks above is equivalent to proving the following:
\begin{enumerate}
\item $b = a^{n+1}$ holds in $LR_n(D'_L)$.
\item $b * c = (a * c)^{n+1}$ holds in $LR_n(D_L)$.
\end{enumerate}

In $LR_n(D'_L)$, 
 \[b * c = (a * c)^{n+1}\]
 It follows from lemma \ref{LR1-2.5} that,
 \[(b * c) = a^{n+1} * c\]
 From the right inverse axiom, 
 \[b = a^{n+1}\]
 which implies that the first relation holds in $LR_n(D'_L)$.
 
 In $LR_n(D_L)$, 
 \[b  = a^{n+1}\]
 By mutiplying on both sides by $c$,
\[b *c =   a^{n+1}*c\]
From lemma \ref{LR1-2.5} it follows that,
\[b *  c = (a * c)^{n+1}\]
Thus the second relation holds in $LR_n(D_L)$ and the the two racks are isomorphic.}\\

{\textbf{Case B: }}{The second move in the Legendrian Reidemeister move of type 2 reduces to the above case, by substituting $b'$ as $b/c$ and $a'$ and $a/c$.}\\

{\textbf{Case C:}} {The third move in the Legendrian Reidemeister move of type 2 affects the presentations of $LR_n(D_L)$ and $LR_n(D'_L)$ in the following manner:
 \begin{enumerate}
 \item Both $LR_n(D_L)$ and $LR_n(D'_L)$ contain the relation $b = a^{n+1}$.
 \item In $LR_n(D'_L)$, some of the relations involving $c$ in $LR_n(D_L)$ are replaced by relations obtained by substituting $(c/a)*b$. These are precisely the relations involving $c$ that occur after the segment which is pushed under the cusp.
 \end{enumerate}

It suffices to show that $(c/a)*b = c$  in $LR_n(D'_L)$ to show isomorphism of $LR_n(D_L)$ and $LR_n(D'_L)$.  By substituting for $b$ in $(c/a)*b$, we obtain
\[(c/a)*b = (c/a) * a^{n+1}\]
Which implies from lemma \ref{LR1-2-2},
\[(c/a) * a^{n+1}= (c/a) * a\]
From the right inverse axiom it follows that,
\[(c/a) * a = c\]
Thus it follows that $(c/a)*b$ equals $c$ in $LR_n(D'_L)$ and the isomorphism of the two racks.}\\

{\textbf{Case D: }}{This Legendrian Reidemeister move is analogous to the earlier one, except that we replace $(c*a)/b$ with $(c*b)/a$ in step above. The proof here consists of proving that $(c*b)/a$ equals $c$ in $LR_n(D'_L)$.
\[(c * b)/a = (c * a^{n+1})/a\]
By applying lemma \ref{LR1-2-2},
\[(c * b)/a = (c * a)/a \]
Existence of right inverse tells us that the right side side is the same as $c$. Thus we obtain invariance of $n$-Legendrian rack under Legendrian Reidemeister moves of type 2.}
\end{caseof}

\begin{itemize}
\item \textbf{Legendrian Reidemeister move 3}
\end{itemize}
This move is the same as in the case of topological knots. The invariance under this move follows from the associativity axiom of racks, as in the case of knots.

This it follows that an $n$-Legendrian rack is an invariant of oriented Legendrian knots.
\end{proof}
\begin{remark}
As a consequence of the above theorem, the notation $$LR_n(L) : = LR_n(D_L) $$
makes sense. Simiarly, we can use the notation $\mathcal{LR}(L)$ for the family of Legendrian racks. Thus, the family $ \mathcal{LR}(L) = \big(LR_n (L)\big)_{n \in \mathbb{N}}$ of Legendrian
racks is an invariant of $L$ upto isotopy through oriented Legendrian knots.
\end{remark}

\section{Legendrian Racks and Legendran Unknots}\label{Leg_unknot}
 
This section contains some basic results about distinguishability of Legendrian racks associated to Legendrian knots. We begin by showing that it distinguishes (topological) knot type. 
\begin{theorem} If the Legendrian knots $L_1$ and $L_2$ correspond to topologically distint knots,
\[\forall n \in \mathbb{N}.\ \ LR_n(L_1) \neq LR_n (L_2)\]
\end{theorem}
\begin{proof}
We begin by noting that quandles are obtained from Legendrian racks by adding the axiom $\forall x.\ x^2 = x$. For any given value of $n \in \mathbb{N}$, consider the racks associated to Legendrian knot by adding $R_i = \{ x^2| x\in LR_n(L_i)\}$ to the set of relations in finite presentation of $LR_n(L_i)$. The resultant racks are quandles $Q(L_i)$ associated to $L_i$. Since quandles are a complete invariant upto mirror-isotopy and orientation, 
\[Q(L_1) \neq Q(L_2)\]
This implies that $LR_n(L_1) \neq LR_n(L_2)$. Since $n$ was arbitrary, the $n$-Legendrian racks are distinct for every choice of $n\in \mathbb{N}$.
\end{proof}

The indexed family of Legendrian racks is not a complete invariant. It can illustrated from the following fact:
See Figure \ref{same-L-rack}
\begin{figure}
\centering
\scalebox{0.5}{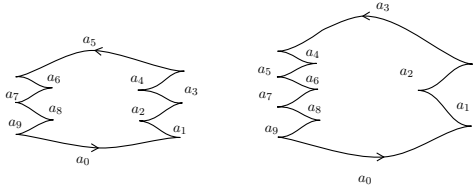}
\caption{Legendrian Knots having isomorphic Legendrian racks. Both the diagrams give rise to the presentation
$ \langle a_0, a_1, a_3, \ldots, a_9 | a_{i+1} = a_i^n \text{ for all } 0 \leq i \leq 9   \rangle $}\label{same-L-rack}
\end{figure}

\begin{theorem} If two topologically trivial Legendrian knots $L_1$ and $L_2$ have the same Thurston-Bennequin invariant, then the corresponding $n$-Legendrian racks are isomorphic, $\forall n \in \mathbb{N}$.
\end{theorem} 
\begin{proof}
It follows from Eliashberg-Fraser's result \cite{Eliash}
, that any Legendrian knot corresponding to the topological unknot has a diagram (in front projection) with no crossings. In such a projection, the writhe equals $0$ and the Thurston-Bennequin number equals the additive inverse of number of cusps. Since \[tb(L_1) = tb(L_2)\]
They have the same number of cusps, in the front projection with no crossings. As a consequence, they have the same number of strands. One can see that the Legendrian rack presentations corresponding to such strands are the same. It thus follows that they are isomorphic.
\end{proof}

\begin{remark} Even though for Legendrian knots corresponding to topological unknots, the racks are independent of the orientation, it is unclear whether such a result would hold true in general. It can be seen that by reversing the orientation, 
\end{remark}

Next few theorems illustrate that these racks  detect some of geometric properties of Legendrian knots, which cannot be distinguished by the knot quandle. We consider the case of topologically trivial Legendrian unknots to illustrate this fact. 

\begin{notation} We refer to a  diagram of a topologically trivial Legendrian knot with no crossings as the minimal diagram. We refer to the topologically trivial Legendrian knot  in Figure 9 as the minimal Legendrian unknot.For  a topologically trivial Legendrian knot $L$, the number of strands in the minimal diagram is unique which denote by $s_{min}(L)$. One may observe that 
 \[s_{min}(L) = 2\mid tb(L)\mid \]  
 \end{notation}
 
\begin{figure}
\centering
\scalebox{.75}{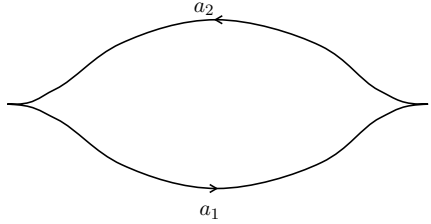}
\caption{minimal Legendrian unknot}\label{std_unknot}
\end{figure}

\begin{theorem} Minimal Legendrian unknot has a trivial $n$-Legendrian rack, for every $n \in \mathbb{N}$.
\end{theorem} \label{minimal-trivial}
\begin{proof}
The $n$-Legendrian rack corresponding to the Legendrian unknot in question has the following presentation:
\begin{itemize}
\item Generators: $a$, $b$.
\item Relations: $b = a^{n+1}$, $a = b^{n+1}$.
\end{itemize}
The above relations imply that:
\[a = ({a^{n+1}})^{n+1}\]
By subtituting on R.H.S from  \ref{LR1-3},
\[ a = a^{2n+1}\]
By substituting on L.H.S from the definition of $n$-Legendrian rack,
\[a^{2n+2} = a^{2n+1}\]
Since right inverse exists in a rack, it follows that:
\[a^{2n+1} = a^{2n}\]
By consecutive application of right inverse axiom, it follows that:
\[a * a = a\]
This implies that every $n$-Legendrian rack is the trivial quandle.
\end{proof}

\begin{lemma}
Let $L$ be a topologically trivial Legendrian knot, such $s_{min}(L)$  is divisible by a prime $p$ greater than $2$.  $LR_n(L)$ is non-trivial, for some $n\in \mathbb{N}$.
\end{lemma}
\begin{proof}

We prove the above theorem by constructing a surjective rack homomorphism from $LR_n$ to the finite Legendrian rack $C_{p^k}$.  Observe that $p^k$ is odd, which implies that there exists an  $n$ such that $p^k = 2n+1$. Let $T$ be a function on $LR_n(L)$ such that:
\[\forall x.\ T(x) = x^{n+1}\]

$LR_n(L)$ can be described by the following presentation:
\[\langle a\ |\ T^{s_{min}(L)} (a) = a \rangle\]

Consider the partial map $\phi$ from $LR_n(L)$ to $C_{p^k}$, such that:
\[\phi(a) = 1\].

From theorem \ref{fin-L-rack}, we know that $C_{p^k}$ is an $n$-Legendrian rack.  The map $\phi$ extends to a well defined homomorphism if and only if:
\[ \phi \left(T^{s_{min}(L)}(a) \right) = 1\]

Observe that for any $i \in C_{p^k}$ and $j \in \mathbb{N}$, 
\[i^{j+1} = (i + j) \ (mod\ p^k)\]
which further implies that,
\[\phi\left(T(a)\right) = (i + n)\ (mod\ p^k)\]
Which generalises to,
\[\phi\left( T^j(a) \right) = (i + nj)\ (mod\ p^k)\]
Thus it follows that,
\[\phi\left( T^{s_{min}(L)}(a) \right) = (1 + ns_{min}(L))\ (mod\ p^k)\]
Since $p_L\divides s_{min}(L)$, it follows that
\[T^{s_{min}(L)}(a) = 1\]
To prove surjectivity, consider the set $\{ \phi(T^j(a))\}_{j < s_{min}(L)}$ in the range of $\phi$. From above we note that it corresponds to the orbit of $i$ by applying the constant shift $n$. Since $n$ and $p^k(=2n+1)$ are co-prime, the orbit spans the rack. This implies that $\phi$ is a surjective homomorphism. 
\end{proof}

One may note that considering a surjective map to $C_p$ suffices for the above case. The $C_{p^k}$ was chosen with a view of the proof of the following result. 

\begin{theorem**}
Let $L_1$ and $L_2$ be two Legendrian unknots. If there exists an odd prime $p$ and a positive integer $k \in \mathbb{N}$, such that:
\begin{itemize}
\item $p^k \divides s_{min}(L_1)$
\item $p^k \ndivides s_{min}(L_2)$
\end{itemize}
Then $\mathcal{LR}(L_1) \neq \mathcal{LR}(L_2)$.
\end{theorem**}\label{Main-Theorem2}
\begin{proof}
We know from the above proof that there exists surjective homomorphism $\phi$ from $LR_n(L_1)$ to $C_{p^k}$. Now we proceed to prove that there does not exist a homomorphism from $LR_n(L_2)$ to $C_{p^k}$. Since $p^k \ndivides s_{min}(L_2)$, there exists positive integers $q$ and $r$ such that:
\[s_{min}(L_2) = p^kq  + r\]
 where $0<r<p^k$. Assume there exists a homomorphism $\psi$ from $LR_n(L_2)$ to $C_{p^k}$, where $a$ maps to some $i\in \mathbb{N}$. This homomorphism has to satisfy the condition on shift operation as above, 
 \[\phi\left( T^{p^kq + r}(a)\right) = i\]
 This implies that,
 \[(i + np^kq + nr)\ (mod\ p^k) = i\]
 Since $p^k \divides np^kq$ and $p^k = 2n+1$ we get,
 \[(i + nr)\ (mod\ 2n+1) = i\]
 But since $n$ and $2n+1$ are co-prime and $r$ is less than $(2n+1)$ and greater than $0$, the above equality cannot hold. Thus we have a contradiction. Thus there does not exist a homomorphism from $LR_n(L_2)$ to $C_{p^k}$.
\end{proof}

\begin{corollary} For any Legendrian unknot $L$ such that $s_{min}(L)$ is divisible by an odd number, \[\mathcal{LR} (L) \neq \mathcal{LR}(\normalfont{\texttt{minimal Legendrian unknot}}).\].
\end{corollary}

One may note that only excluded cases from Theorem \ref{Main-Theorem2} are the cases where $s_{min}(L_1) = 2^m s_{min}(L_2)$, for some $m\in \mathbb{N}$.

\section{Motivation and Alternate description}\label{alt_defn}

In an $n$-Legendrian rack, one could define a predicate $C(a, b)$, such that 
\[C(a, b) \iff b = a^{n+1}\]
Such a relation proves useful to describe the relationship between strands related by a cusp. Ideally, it is desirable for an invariant of Legendrian knot to be able to atleast detect the difference between the number of `up' and `down' cusps in standard projections of the oriented Legendrian knot. 

 Natural question that arises is whether it would it be to possible to obtain a  rack invariant of oriented Legendrian knots, by alloting distinct relations to up and down cusps, which distinguishes the `type' of cusp. If such an invariant exists, it would be possible to define predicates $U(a, b)$ and $D(a, b)$, which relate strands joined together by an up and down cusp respectively.  

In this section, we demonstrate the difficulty of obtaining such an invariant in terms of racks. In more precise terms, we show that any invariant of an oriented Legendrian knot such that:
\begin{itemize}
\item Axioms of the invariant are expressible in first order logic.
\item The invariant satisfies the rack axioms.
\item All the axioms are universally quantified.
\end{itemize}
is equivalent to some $n$-Legendrian rack, and thus fails to distinguish cusps.  Infact, the result holds true even when both $U$ and $D$ are not definable in terms of $(*)$. 

This result is shown along the following lines:
\begin{itemize}
\item We construct the axioms for such an invariant in a langauge with signature $(*, U, D)$. 
\item We axiomatize Legendrian Reidemeister moves and rack axioms in this langauge.
\item We show that in such a axiom schema:
\begin{enumerate}
\item $U \equiv D$.
\item $U(a, b) \iff b = a^{n+1}$.
\item $a = a^{2n+2}$.
\end{enumerate}
\item We further demonstrate the equivalence of this with the $n$-Legendrian rack, by showing that not only are axioms of $n$-Legendrian racks, are a consequence of these axioms, but the converse also holds true.
\end{itemize}

\begin{definition} An $n$-Legendrian predicate rack $\mathbb{LR}_n$ is a rack along with two binary predicates $U$ and $D$, such that following axioms are satisfied:

  \vspace{4pt}
   \begin{enumerate}

   \item {LR Move 1:}

   \begin{enumerate}[label =(\alph*)]
\item 
$\forall x.\ \forall y.\ \forall z.\ (U(x*z, y) \wedge D(y, z))  \longleftrightarrow ((x = z) \wedge	(y = (x^{n+2}))$.

\item
$\forall x.\ \forall y.\ \forall z.\ (D(y, z/x) \wedge U(x, y))   \longleftrightarrow ((x = z) \wedge (y = x^{-(n+2)}))$.


\item 
 $\forall x.\ \forall y.\ \forall z.\  (D(x, y) \wedge U(y, z/x))  \longleftrightarrow ((x = z) \wedge (y = (x^{-(n+2)}))$.


\item
$\forall x.\ \forall y.\ \forall z.\ ((D(x*z, y) \wedge U(y, z))   \longleftrightarrow ((x = z) \wedge (y = x^{n+2}))$.

\end{enumerate}
 
 \vspace{4pt}
\item LR Move 2:
\begin{enumerate}[label = (\alph*)]
\item 
 $\forall x\ \forall y\ \forall z.\ (U(x,y)  \longleftrightarrow  U(x*z, y*z))$.

\item 
$\forall x\ \forall y\ \forall z.\ (D(x,y)  \longleftrightarrow D(x*z, y*z))$.

\item 
 $\forall x\ \forall y\ \forall z.\ (U(x, y)  \longrightarrow (z/x)*y = z)$.

\item 
 $\forall x\ \forall y\ \forall z.\ (D(x,y)  \longrightarrow (z*y)/ x = z)$. 

 \end{enumerate}
\end{enumerate} 
 \end{definition}

To every oriented Legendrian knot diagram (in the front projection), we associate a $n$-Legendrian predicate rack in a manner similar to the $n$-Legendrian racks, except that the relation between two strands $a$ and $b$ meeting at a cusp is replaced by the predicate relation $U(a,b)$ or $D(a, b)$. These choice of predicate relation depends on whether the we traverse in upward or downward directions when going from $a$ to $b$. This is illustrated in the Figure \ref{U_D}. The $n$-Legendrian predicate-rack associated to a Legendrian knot $L$ is the Legendrian predicate-rack generated by the associated relations.

\begin{figure}
\centering
\scalebox{0.75}{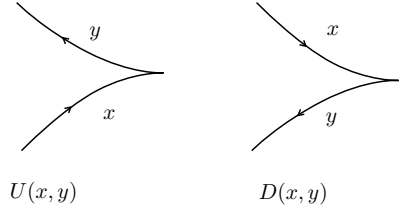}
\caption{Predicates $U$ and $D_L$.}\label{U_D}
\end{figure}

The axioms subsequently derive from the Legendrian Reidemeister moves as illustrated by the Figure \ref{lraxioms}. One can check that they arise as a formulation of the Legendrian Reidemeister moves, though some further justification needs to be provided for $y = x^{n+2}$ and $y = x^{-(n+2)}$ terms in LR Move 1.

\begin{figure}
\centering
\scalebox{0.6}{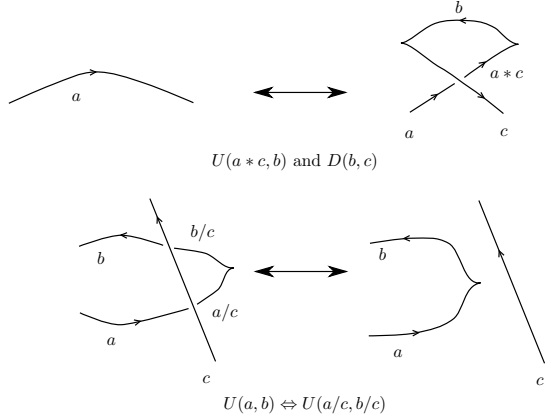}
\caption{Examples of LR axioms}\label{lraxioms}
\end{figure}

We label the figures on left and right, in each of the above examples as $D_L$ and $D'_L$ respectively. Observe that $b$ and $c$ in the move corresponding to the Legendrian Reidemeister move 1 in the above figure, has to belong to the Legendrian predicate rack $\mathbb{LR}_n(D)$, in order for $\mathbb{LR}_n$ to be an invariant. One might further note that $c$ has to equal $a$ to ensure preservation of relations involving the strand corresponding to $a$. The most natural method of ensuring that $b$ also belong to $\mathbb{LR}_n$ is to equate to some product of $a$'s. The following lemma tells us that such a product will eventually equal some $a^n$. 

\begin{lemma} Given an element $x$ in $R$ such that $x$ can be expressed only in terms of $a$, where $a \in R$. Then there exists $m \in \mathbb{Z}_{+}$ such that:
\[ x = a^m\]
\end{lemma}\label{rewriting}

\begin{figure}[h]
\centering
\scalebox{0.75}{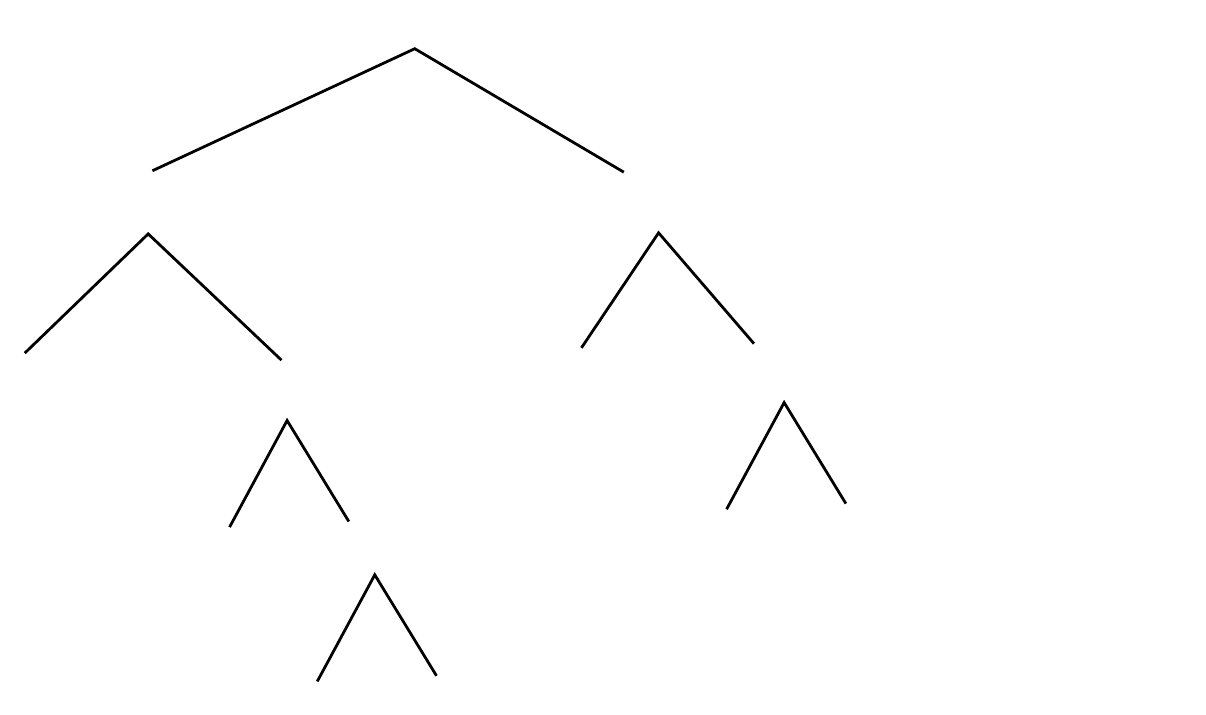}
\end{figure}

\begin{proof}
We prove it be induction on least number of $a$'s in a word representing $x$. It is easy to see that above holds true for 
the cases 1 and 2. We assume that it holds true for all $n \leq k$. We show that it holds true for $n = k+1$. We begin by 
observing that every term consisting of $a$'s be represented by a rooted full binary tree whose leaves consist of $a$, and 
nodes denote terms obtained by multiplying the children of the node. We can further label the edges $L$ and $R$ to denote 
the order of multiplication of children. Consider the root of the binary tree. It has two attached children. Each of them 
denote a term with wordlength equal to or less than $k$. By applying induction hypothesis, we obtain the following 
statement:

\[\exists m_1 \in \mathbb{Z}_{+}.\exists m_2\in \mathbb{Z}_{+}.\ (x = a^{m_1} * a^{m_2})\]
By applying lemma \ref{LR1-2-2} ,
\[x = a^{m_1}*a\]
Thus proved.
\end{proof}
 Thus it follows that $b$ should equal $a^n$, for some $n \in \mathbb{Z}_+$. Analogous results can be proved for words consisting of $a^{-1}$. If $b = a$ or $b = a^2$, it can be shown that the structure reduces to a quandle. For  $n \geq 3$, it gives us a rack invariant, as illustrated in the following theorems. 
 
The rack obtained from the generators modulo generating relations give rise to the the domain set. Now the corresponding functions $U$ and $D$ are obtained by considering the  minimal boolean valued functions generated by the cusps which satisfy the above axioms. Existence of such a minimal function follows from Godel's completeness theorem. .\\
\begin{remark} It maybe interesting to note that the $n$-Legendrian predicate rack defined above, constituted our original formulation of the oriented Legendrian knot invariant. Some of the observations in the computer generated proofs, enabled the simpler structure and definition presented in Section \ref{Leg_racL_defn}.
\end{remark}

\subsection{Proof of Equivalence}
We begin by proving that every $n$-Legendrian predicate-rack is an $n$-Legendrian rack.

\begin{theorem}\label{LPRimpliesLR} Every $n$-Legendrian predicate-rack is an $n$-Legendrian rack.
\end{theorem}
\begin{proof}
LR Move 1(a) can be written in disjunctive normal form after relabelling as,
\[ \forall x. \forall y. \forall z. \ \neg U(x\ * y, z) \vee \neg D(z, y) \vee (x = y) \addtag \]
This can be rewritten after substituting $x_1$ for $x * y$, $y_1$ for $z$ and $z_1$ for $y$.
\[\forall {x_1} . \forall y_1. \forall z_1.\ \neg U(x_1, y_1) \vee \neg D(y_1, z_1) \vee (x_1/z_1 = z_1) \addtag  \]
For $x_1 = a^{n+2}$, $y_1 = a$ and $z_1 = a^{-(n+2)}$,
\[\neg U (a^{n+2}, a) \vee \neg D(a, a^{-(n+2)}) \vee (a^{n+2}/a^{-(n+2)} =  a^{-(n+2)} ) \addtag  \]

Observe that it follows from LR Move-1(D) that by substituting for $y$ and $z$ in $U(y,z)$,
\[\forall x.\ U(x^{n+2},x) \addtag \] 
Similarly it follows from LR Move-1(c) by substituting for $y$ in $D(x,y)$,
\[\forall x.\ D(x,x^{-(n+2)}) \addtag \]
It follows from (4.3),(4.4) and (4.5), 
\[a^{n+2}/a^{-(n+2)} =  a^{-(n+2)}  \addtag \]

Similarly LR Move 2(d), gives us:
\[\forall x. \forall y.\ \neg D(x, y) \vee ((z*x)/y = z). \addtag \]

From (4.4), we get
\[\forall x. \forall y. \forall z. ((z*x)/y = z) \addtag \]

For $x = a$, $y = a^{-(n+2)}$ and $z = a^{n+1}$, the above equation reduces to,
\[
\addtag
(a^{n+1}*a)/a^{-(n+2)} = a^{n+1} 
\]
Applying the above in (4.6), we get:
\[a^{-(n+2)} = a^{n+1} \addtag \]
From Lemma \ref{LR1-4}, 
\[a = a^{2n+2} \addtag \]
\end{proof}

The above theorem tells us that even $n$-Legendrian predicate-rack is also a $n$-Legendrian rack. We need the following theorems to establish that $n$-Legendrian predicate-rack corresponding to two oriented Legendrian knots are the same:

\begin{theorem}\label{cusp_equiv_nth_power} In a predicate Legendrian predicate-rack $\mathbb{LR}_n$, the following holds true:
\[\forall x.\ \forall y.\ D(a,b) \iff b = a^{n+1}\]
\end{theorem}
\begin{proof}
Assume that $b=a^{n+1}$. From Lemma \ref{LR1-4},
\[x^{n+1} =  x^{-(n+2)}\]
As observed in the (4.5) of the previous theorem, $D(x, x^{-(n+2)})$ holds true for all $x$. It follows that  $D(x, x^{n+1})$. Thus $D(a, b)$ holds.
Now we prove the converse, assume $D(a,b)$. This implies from the $n$-Legendrian Reidemeister axiom 2(a) that:
\[\forall x. (x*a)/b = x \addtag \]
By substituting $x'/a$ for $x$ in above, we get:
\[\forall x'.\ (x'/b) = (x'/a) \addtag \]
For $x' = a^{n+2}$ we get,
\[a^{n+2}/b = a^{n+1} \addtag\]

From  (4.4) of the previous theorem, $U(a,a^{n+2})$,, and from our assumtion that $D(a,b)$ holds true. By substituting in (4.2) of the previous theorem we get,
\[ a^{n+2}/b = b \addtag \]
Using (4.14), we get
\[b = a^{n+1} \addtag \]
Thus proved.
\end{proof}

\begin{theorem} In a predicate Legendrian predicate-rack $\mathbb{LR}_n$, the following holds true:
\[\forall x.\ \forall y.\ U(a,b) \iff\ b = a^{n+1}\]
\end{theorem}
\begin{proof}
Axioms of Legendrian predicate-racks are symmetric in $U$ and $D$, which means that the axioms remain unchanged even if we interchange $U$ and $D$ in the axioms. This implies that by replacing $D$ with $U$ in proof of , the proof still holds true. The result thus follows.
\end{proof}

\begin{corollary} The following holds in every $n$-Legendrian predicate rack.
\[\forall x.\forall y.\ U(x, y) \iff D(x, y)\]
\end{corollary}

From the above theorem and corollary, it follows that by substituting for $U$ and $D_L$ in the presentation of an $n$-Legendrian predicate-rack corresponding to the front projection of Legendrian knot $L$, one obtains the presentation of the corresponding $n$-Legendrian rack.  For the sake of completeness of the equivalence result, we show that the axioms of $n$-Legendrian predicate-racks are satisfied by $n$-Legendrian racks.

\begin{theorem} Consider an $n$-Legendrian rack $LR_n$ with the following choice of predicate:
\[U(x, y) \longleftrightarrow (y = x^{n+1})\]
then it satisfies the $n$-Legendrian predicate axioms.
\end{theorem}
\begin{proof}
By substituting for the $U$ and $D_L$ in $LR_n$, we obtain the same set of equalities that were proved in the theorem \ref{lracL_invariance}, during the course of proving invariance under Legendrian Reidemeister moves. Since these equalities hold true in $LR_n$, it follows that that $LR_n$ is a Legendrian predicate-rack. 
\end{proof}

The following corollary follows from the equivalence of Legendrian predicate-racks and Legendrian racks.
\begin{corollary} Each $n$-Legendrian predicate-rack is an $n$-Legendrian rack.
\end{corollary}

\section{Experimenting with Automated Theorem Proving}\label{fotp}

Automated theorem provers are computer programs which check if a given statement is a logical consequence of a set of axioms/definitions. They either attempt to prove a statement using a technique called resolution or they construct counterexamples using finite models of the axioms. Automated theorem provers for first-order logic, deal with axioms and statements expressible in first order logic. The previous use of automated theorem proving for unknot recognition using quandles described in\cite{Fish},  motivated the authors to carry out a similar exercise for Legendrian knots. The obstruction to undertaking such an exercise was the lack of any known rack invariants for Legendrian knots.  

We then carried out the following sequence of experiments in the automated theorem prover -Prover9 and Mace4\cite{Prover9},  which aided the discovery of the rack invariants described in the paper:
\begin{enumerate}

\item The set of axioms, resembling the axioms listed in Section \ref{alt_defn}, were used as a candidate first-order logic formalization of Legendrian knots. The set of axioms that were employed differed only in the respect that we used $y = x$ instead of the $(y = x^{n+2})$ and $(y = x^{-(n+2)})$  terms in the LR move 1.

\item The above mentioned axiomatization ended up proving that 
\[U(a,b) \rightarrow a = b\]
This implied that these racks were equivalent to the topological quandle. We subsequently replaced the term  $(y = x)$ in LR moves 1 with $(y = x^m)$, and varied the value of $m$ till Mace4 was able to produce a counter example for the assertion:
\[U(a,b) \rightarrow a = b\]
\item We further carried out experiments to test various conjectures about triviality of racks for different Legendrian unknots and for different values of $m$. Some of the observations in subsequently generated proofs guided us towards the equivalence of $U$ and $D$, as well as their definability in terms of the rack operation $(*)$.

\item The definability of $U$ and $D$ in terms of $(*)$,  led to the current formulation of $n$-Legendrian racks. It also involved some further refinements (by human means) to bring it to the minimal possible expression, in terms of number of axioms and wordlength of the rack elements.\end{enumerate}

Subsequently, Prover9 and Mace4 were used to derive and test conjectures about some of the $n$-Legendrian racks associated to Legendrian knots. Based on observations in computer generated proofs, more general proofs were obtained which held true for  arbitrary Legendrian racks. In some of the cases, the proofs produced in the paper are near-faithful human translations of computer-generated proofs, such as the proof of lemma \ref{LR1-1} and theorem \ref{LPRimpliesLR}. In some of the cases, observations from computer generated proofs guided us to many interesting lemmas, which were subsequently used to obtain proofs shorter than the machine generated proofs of the theorems. Main Theorem 1 and the lemmas preceeding it largely fit into this category. In some of the cases, our proofs were completely independent of the computer generated proofs. This includes many of the subcases in Main theorem 1, and theorem \ref{minimal-trivial} . Even in these cases, the automated theorem prover still provided an experimental confirmation of our conjectures. Computer generated finite models were also used to derive the finite rack invariants $C_k$. Observations based on these counter-examples, enabled the more general proof of Main Theorem 2.
One of the other applications of an automated theorem prover is its utility in distinguishing a Legendrian knot from the minimal Legendrian unknot. It can be achieved along on the following lines:
 
 \begin{enumerate}
 \item Input the axioms of an $n$-Legendrian racks and specify the presentation of the Legendrian rack as axioms.
 \item Check if the following statement returns a  counter example in a finite model: 
 \[a_1 = a_2 = a_3 = ... = a_n\]
Where $a_i$'s are the generators.
\item If the above is true, check if $a_1 * a_1 = a_1$.
\item If the rack is non trivial, it returns a counter example through a finite model. Non triviality of the rack implies that it is not the minimal Legendrian unknot.
\item If it fails, repeat it for the axioms corresponding to $n$-Legendrian rack, with a different value of $n$.
\end{enumerate}

These experiments enabled the following conjecture:
\vspace{10pt}
\begin{conjecture} For a Legendrian unknot $L$, such that:
\[\mid tb(L) \mid =  2^k\]
for some $k\in \mathbb{N}$, every $n$-Legendrian rack is trivial.
\end{conjecture}

This conjecture arises out of experimental results obtained out of checking triviality for 4,6,8 and 32 cusped Legendrian unknots. For each of the cases above,   the $n$-Legendrian racks, where $n$ varies from 1 to 8, were proved by the automated theorem prover to be trivial. Extracting a general proof technique from the computer generated proofs was fairly difficult, since none of the proofs easily lent themselves to an obvious generalization. One can gauge that by looking at the explosion in proof size, as illustrated by the following data:\\
\begin{center}
\begin{longtable}{*{2}{m{0.5\linewidth}}}

\begin{tabular}{| c | c | c | c | c |}
\hline
 & 4-cusp & 8-cusp & 16-cusp & 32-cusp\\
\hline
1 & 46 & 96 & 146 &	233\\
\hline
2 & 78 &	 98 & 223 & 282\\
\hline
3 & 119	& 202 &	363 & 407\\
\hline
4 & 66	& 202 & 222 & 551\\
\hline
5 & 102 & 181 &	338	& 518\\
\hline
6 & 102	& 139 & 245 & 431\\
\hline
7 & 119	& 146 & 253 & 427\\
\hline
8 & 341 & 445 & 453	& 2017\\
\hline 
\end{tabular}
 
\end{longtable}
\end{center}
\vspace{3pt}

Since automated theorem provers do not lend themselves easily to checking rack isomorphims, for non-trivial racks, we could not obtain any experimental results for general Legendrian unknots. However, by generalizing from the above experimental data and underlying assumptions in the proof of Main Theorem 2, we postulate the following conjecture:
\begin{conjecture}
For Legendrian unknots $L_1$ and $L_2$, such that:
\[tb(L_1)  =  2^k tb(L_2)\]
for some $k\in \mathbb{Z}_{\geq 0}$, 
\[\mathcal{LR}(L_1) = \mathcal{LR}(L_2).\]
\end{conjecture}

The computer generated proofs and counter-examples are available on:
\begin{center}
\url{https://github.com/prathamesh-t/Legendrian-Racks}
\end{center}

\section{Conclusion and Further Questions}\label{concl}

In this work, we introduced the rack invariants of Legendrian knots and demonstrated their ability to distinguish some of the Legendrian knots. It leaves several questions wide open. Some of the more basic questions about Legendrian racks include:

Whether these racks indeed detect orientation? Would it be possible to expect analogues of Main Theorem 2 for Legendrian knots which are not topologically trivial?
 
A more geometric interpretation of the Legendrian rack structures would be fairly desirable. Perhaps the connection to Thurston-Bennequin number might shed some light in this direction. A more general result that could precisely formulate the link between Legendrian rack and Thurston-Bennequin number might be possible and very desirable. 

On the more algebraic side, relationship with other algebraic structures such as groups, analogous to the case of quandles and conjugacy action on groups, would prove fairly useful. Such an analogue, might enable a possible proof about the decidability of the word problem on Legendrian racks. Possibilities of co-cycle invariants and rack homologies and whether they lead to any known or new useful invariants, is another question that remains to be explored. 
  
\bibliography{references}

\end{document}